\documentclass[12pt]{amsproc}
\usepackage{amssymb,amsthm,color,qtree,nameref,hyperref,cleveref}
\usepackage{amssymb,amsthm,amsmath,latexsym,url, amsfonts}
\usepackage{enumitem}
\usepackage{adjustbox}
\usepackage{thmtools,mathtools}
\usepackage{mathrsfs}

\newtheorem{theorem}{Theorem}
\newtheorem{lemma}[theorem]{Lemma}
\newtheorem*{thmA}{Theorem A}
\newtheorem*{thmB}{Theorem B}

\newtheorem{proposition}[theorem]{Proposition}
\newtheorem{qst}[theorem]{Question}
\newtheorem{corollary}[theorem]{Corollary}

\newtheorem{conjecture}[theorem]{Conjecture}

\DeclareMathOperator*{\Psl}{PSL}

\DeclareMathOperator*{\Sz}{Sz}
\def\d{\delta^*}
\def\e{\varepsilon}
\def\g{\gamma^*}
\def\gi{\gamma_{\infty}}

\begin{document} 

\author[Bastos]{ Raimundo Bastos}
\address{ Department of Mathematics, University of Brasilia,
Brasilia-DF, 70910-900 Brazil }
\email{bastos@mat.unb.br}
\author[Monetta]{ Carmine Monetta }
\address{Dipartimento di Matematica, Universit\`a di Salerno, Via Giovanni Paolo II, 132 - 84084 - Fisciano (SA), Italy}
\email{cmonetta@unisa.it}


\title[Coprime commutators]{Coprime commutators in finite groups}

\maketitle

\begin{abstract}
Let $G$ be a finite group and let $k \geq 2$. 
We prove that the coprime subgroup $\g_k(G)$ is nilpotent if and only if $|xy|=|x||y|$ for any $\g_k$-commutators $x,y \in G$ of coprime orders (Theorem A). Moreover, we show that the coprime subgroup $\d_k(G)$ is nilpotent if and only if $|ab|=|a||b|$ for any powers of $\d_k$-commutators $a,b\in G$ of coprime orders (Theorem B).

\end{abstract}

\section{Introduction}

In finite Group Theory, the nilpotency of a group can be analyzed looking at the behavior of its elements of coprime orders. Indeed, in \cite{baubau} B.\,Baumslag and J.\,Wiegold proved the following result.

\begin{theorem}[\cite{baubau}]\label{teobau} Let $G$ be a finite group in which $|ab|=|a||b|$ whenever the elements $a,b$ have coprime orders. Then $G$ is nilpotent.
\end{theorem}

 Here the symbol $|x|$ stands for the order of the element $x$ in a group $G$. In this direction, in 2016, R.\,Bastos and P.\,Shumyatsky established a characterization for the nilpotency of the commutator subgroup  of a finite group \cite{BS}.

\begin{theorem}[\cite{BS}]\label{teoBS} Let $G$ be a finite group in which $|ab|=|a||b|$ whenever the elements $a,b$ are commutators of coprime orders. Then $G'$ is nilpotent.
\end{theorem}

 This theorem represents a first successful attempt to broaden Theorem~\ref{teobau} to the world of group-words. Recall that a {\it word} (or {\it group-word}) $w=w(x_1, \ldots, x_k)$ is a nontrivial element of the free group $F=F(x_1, \ldots, x_k)$ on free generators $x_1, \ldots, x_k$. A word is a {\it commutator word} if it belongs to the commutator subgroup $F'$ of $F$. 
 
 Given a word $w$ and a group $G$, we denote by $G_w$ the set of all $w$-values, and the subgroup of $G$ generated by $G_w$ is called the {\it verbal subgroup of $G$ corresponding to $w$}, and it is usually denoted by $w(G)$.  An example of commutator word is the $k$th lower central word $\gamma_k=~\gamma_k(x_1, \ldots, x_k)$ defined inductively by the formulae
\[
\gamma_1=x_1
\qquad \text{and} \qquad
\gamma_{k+1}=[\gamma_{k},x_{k+1}]
\quad
\text{for $k\ge 1$.}
\]

 As one can aspect, the verbal subgroup corresponding to $\gamma_k$ is the familiar $k$th term of the lower central series $\gamma_k(G)$, and the $\gamma_k$-values in $G$ are called $\gamma_k$-commutators.

Using the language of group-words, Theorem~\ref{teobau} and Theorem~\ref{teoBS}  show that the  nilpotency of the verbal subgroups $G$ and $G'$ can be obtained imposing the condition about orders on the sets $G_x$ and $G_{[x,y]}$, respectively. This suggests the question whether a similar criterion holds for other group-words.

\begin{qst}\label{qGen}
Let $w$ be a group-word and let $G$ be a finite group in which $|ab|=|a||b|$ whenever the $w$-values $a,b$ have coprime orders. Is then $w(G)$ necessarily nilpotent?
\end{qst}

Easy counterexamples show that Question \ref{qGen} has negative answer in general. Indeed, one can consider the symmetric group $G$ of degree $3$ and the word $w=x^3$. Then the nontrivial $w$-values have order $2$, but $w(G)=G$.

Also in the case of commutator words Question \ref{qGen} has negative answer since one can consider the alternating group $\mathbb{A}_5$ of degree $5$ and the word $[x,y^{10},y^{10},y^{10}]$ all of whose nontrivial values in $\mathbb{A}_5$ have order $2$ (see \cite{MT}).

Despite all the underlined negative answers, C.\,Monetta and A.\,Tortora proved that there is a family of commutator words which answers in the affermative to Question~\ref{qGen}. 

\begin{theorem}[\cite{MT}]\label{teoMT}
Let $w=w(x_1, \ldots, x_k)$ be the simple commutator word defined by 
\[w=[x_{i_1}, \ldots, x_{i_n}],\]
where $i_{1} \neq i_j$ for every $j \in \{2, \ldots, n\}$ and $i_j \in \{1, \ldots, k\}$ for every $j$. If $G$ is a finite group, then $w(G)$ is nilpotent if and only if $|ab|=|a||b|$ for any $w$-values $a,b\in G$ of coprime orders. 
\end{theorem}

Notice that $\gamma_n$-commutators are particular instances of simple commutator words.
Indeed, Theorem~\ref{teoMT} extends the following result proved by R.\,Bastos, C.\,Monetta and P.\,Shumyatsky in \cite{bms}.

\begin{theorem}[\cite{bms}]\label{metanilp} Let $k$ be a positive integer. The $k$th term of the lower central series of a finite group $G$ is nilpotent if and only if $|ab|=|a||b|$ for any $\gamma_k$-commutators $a,b\in G$ of coprime orders.
\end{theorem}

This result contributes towards a positive answer to the following conjecture made in \cite{BS}, which remains still an open problem.

\begin{conjecture}\label{shumy} Let $w$ be a multilinear commutator word and let $G$ be a finite group in which $|ab|=|a||b|$ whenever the $w$-values $a,b$ have coprime orders. Then $w(G)$ is nilpotent.
\end{conjecture}

Recall that a multilinear commutator word is a word obtained nesting commutators but using always different variables. However, not all simple commutator words are multilinear. Indeed, another example of simple commutator words is the $k$th Engel word $\e_k=~\e_k(x,y)$ defined inductively by the formulae
\[
\e_1=[x,y]
\qquad \text{and} \qquad
\e_{k+1}=[\e_k,y]=[x, \overset{k}{y, \ldots, y}]
\quad
\text{for $k\ge 1$.}
\]

Therefore, from Theorem~\ref{teoMT} it follows that Question~\ref{qGen} has a positive answer for Engel words, too.

\bigskip
 
A similar question could be asked for commutators of elements of coprime orders, namely for the coprime commutators $\g_k$ and $\d_k$.

The coprime commutators $\g_k$ and $\d_k$  were introduced by P.\,Shu\-myatsky in \cite{pavel} as a tool to study the nilpotency and the solubility of finite groups in terms of commutators of elements of coprime orders. They are defined as follows. 

Let $G$ be a finite group. Every element of $G$ is both a $\g_1$-commutator and a $\d_0$-commutator. Now let $k \geq 2$ and let $X$ be the set of all elements of $G$ that are powers of $\g_{k-1}$-commutators. An element $g \in G$ is a $\g_k$-commutator if there exist $a \in X$ and $b \in G$ such that $g=[a,b]$ and $(|a|,|b|)=1$. For $k \geq 1$ let $Y$ be the set of all elements of $G$ that are powers of $\d_{k-1}$-commutators. An element $g \in G$ is a $\d_k$-commutator if there exist $a, b \in Y$ such that $g = [a, b]$ and $(|a|, |b|) = 1$. The subgroups of G generated by all $\g_k$-commutators and all $\d_k$-commutators are denoted by $\g_k(G)$ and $\d_k(G)$, respectively. 

One can easily see that if $N$ is a normal subgroup of $G$ and $x$ is an element of $G$ whose image in $G/N$ is a $\g_k$-commutator (respectively a $\d_k$-commutator), then there exists a $\g_k$-commutator (respectively a $\d_k$-commutator) $y$ in $G$ such that $x \in yN$. 

The commutators $\g_k$ and $\d_k$ are closely related to the commutators $\gamma_k$ and $\delta_k$, that is, they control the nilpotency and the solubility of a group as $\gamma_k$ and $\delta_k$ do. Indeed, in \cite{pavel}, it was shown that $\g_k(G)=1$ if and only if $G$ is nilpotent, while $\d_k(G)=1$ if and only if $G$ is soluble and the Fitting height of $G$ is at most $k$. 

Coprime commutators have been studied by many authors (for example see \cite{as, ast, FC, monak, ps}).  In \cite{monak} it is established that the nilpotent residual $\gi(G)$ of a finite group $G$ is generated by commutators of primary elements of coprime orders, where a primary element is an element of prime power order. Moreover, in the same paper, V.\,S.\,Monakhov proved that the nilpotent residual $\gi(G)$ of a finite soluble group $G$ is nilpotent if and only if $|ab| \geq |a| |b|$ for every $a$ and $b$ commutators of primary elements of coprime orders.

Similarly, in \cite{FC} A.\,Freitas de Andrade and A.\,Carrazedo Dantas studied the nilpotency of $\gi(G)$ looking at the orders of powers of $\d_1$-commutators, where $G$ is a finite group. More precisely they proved that the nilpotent residual $\gi(G)$ of a finite group $G$ is nilpotent provided that $|ab| = |a||b|$ whenever $a,b$ are powers of $\d_1$-commutators of coprime orders. In \cite{FC}, they also posed the following question.

\begin{qst}\label{alexi}
Let $k$ be a non-negative integer and let $G$ be a finite group in which $|ab| = |a||b|$ whenever the elements $a,b$ are (powers of) $\d_k$-commutators of coprime orders. Is then the subgroup $\d_k(G)$ nilpotent?
\end{qst}

\noindent When $k=0$, the answer is positive by Theorem~\ref{teobau}.
At the same time, one may think to restate Question~\ref{alexi} in terms of $\gamma_k^*$-commutators. 

Our first main result is the following, which gives a complete answer for $\g_k$-commutators.

\begin{thmA}\label{mainGk}
Let $G$ be a finite group and let $k \geq 2$. Then $\gamma_k^*(G)$ is nilpotent if and only if $|ab|=|a| |b|$ whenever $a$ and $b$ are $\gamma_k^*$-commutators such that $(|a|,|b|)=1$.
\end{thmA}

 In particular, Theorem A applies to $\gi(G)$, showing that the nilpotent residual $\gi(G)$ of a finite group $G$ is nilpotent if and only if $|ab|=|a| |b|$ whenever $a$ and $b$ are $\gamma_k^*$-commutators of coprime orders, for some $k \geq 2$. Moreover, since $\g_2$-commutators and $\d_1$-commutators coincide, Theorem A shows that it is enough to impose the condition on the orders of $\d_1$-commutators without considering their powers, and that a group satisfying such a condition is necessarily soluble, improving both results in \cite{monak} and \cite{FC}.

Concerning the case when $k \geq 2$, we answer in the positive to Question~\ref{alexi} imposing the condition on the orders of powers of $\d_k$-commutators. Indeed we succeed in proving the following result.

\begin{thmB}\label{mainDk}
Let $G$ be a finite group and let $k \geq 2$. Then $\d_k(G)$ is nilpotent if and only if $|ab|=|a| |b|$ whenever $a$ and $b$ are powers of $\d_k$-commutators such that $(|a|,|b|)=1$.
\end{thmB}

\section{Preliminaries}
As usual, if $\pi$ is a set of primes, we denote by $\pi'$ the set of all primes that do not belong to $\pi$. For a group $G$ we denote by $\pi(G)$ the set of primes dividing the order of $G$, and the maximal normal $\pi$-subgroup of $G$ is denoted by $O_{\pi}(G)$. For elements $x,y$ of a group $G$ we write $[x, \ _0y]=x$ and $[x,\ _{i+1}y]=[[x,\ _{i}y],y]$ for $i\geq0$.
We denote by $F(G)$ the Fitting subgroup of~$G$. 

Let $G$ be a soluble group. The {\it upper Fitting series} of $G$ is the series
\[
1=F_0(G) \leq F_1(G) \leq \cdots \leq F_t(G)=G
\]
where $F_1(G)=F(G)$ and $F_{i+1}(G)/F_i(G)=F(G/F_i)$ for every $i \geq 1$. 
Similarly, the  {\it lower Fitting series} of $G$ is the series 
\[
G=N_0(G) \geq N_1(G) \geq \cdots \geq N_t(G)=1
\] 
where $N_{i}(G)=\gi(N_{i-1}(G))$ for every $i \geq 1$. The number $t$ is called the {\it Fitting height} of the group and is denoted by $h(G)$.

Throughout the article we use without special references the well-known properties of coprime actions:  if $\alpha$ is an automorphism of a finite group $G$ of coprime order, $(|\alpha|,|G|)=1$, then $C_{G/N}(\alpha)=C_G(\alpha)N/N$ for any $\alpha$-invariant normal subgroup $N$, the equality $[G,\alpha]=[[G,\alpha],\alpha]$ holds, and if $G$ is in addition abelian, then $G=[G,\alpha]\times C_G(\alpha)$. Here $[G,\alpha]$ is the subgroup of $G$ generated by the elements of the form $g^{-1}g^\alpha$, where $g\in G$.

For convenience of the reader we state some results which will be used later.

\begin{lemma}[\cite{pavel}, Lemma 2.4]\label{uuu}
Let $G$ be a finite group and $y_1, \ldots, y_k$ be $\d_k$-commutators in $G$. Suppose the elements  $y_1, \ldots, y_k$  normalize a subgroup $N$ of $G$ such that $(|y_i|, |N|)=1$ for every $i=1,\ldots, k$. Then for every $x \in N$ the element $[x,  y_1, \ldots, y_k ]$ is a $\d_{k+1}$-commutator.
\end{lemma}

Let us call a subgroup $H$ of a group $G$ a tower of height $h$ if $H$ can be written as a product $H=P_1\cdots P_h$, where

\begin{enumerate}
\item $P_i$ is a $p_i$-group ($p_i$ a prime) for $i=1,\dots,h$.
\item $P_i$ normalizes $P_j$ for $i<j$.
\item $[P_i,P_{i-1}]=P_i$ for $i=2,\dots,h$.
\end{enumerate}
It follows from (3) that $p_i\neq p_{i+1}$ for $i=1,\dots,h-1$. In \cite{turull}, it was proved that a finite soluble group $G$ has Fitting height at least $h$ if and only if $G$ possesses a tower of height $h$.

\begin{lemma}[\cite{pavel}, Lemma 2.6]\label{coprimetower}
Let $P_1 \cdots P_h$ be a tower of height $h$. For every $1 \leq i \leq h$ the subgroup $P_i$ is generated by $\d_{i-1}$-commutators contained in $P_i$.
\end{lemma}
The following theorem shows how $\d_k$-commutators control the solubility of a finite group.

\begin{theorem}[\cite{pavel}, Theorem 2.7]\label{fittingdelta}
Let $G$ be a finite group and let $k$ be a positive integer. Then $\d_k(G)=1$ if and only if $G$ is soluble with Fitting height at most $k$.
\end{theorem}

As an immediate consequence of Theorem~\ref{fittingdelta}, the next result shows the relationship between the terms of the lower Fitting series and  $\d_k$-commutators.

\begin{lemma}\label{deltagen}
If $G$ is a finite group, then the $k$th term $N_k(G)$ of the lower Fitting series of $G$ equals $\d_{k}(G)$ for every $k \geq 0$.
\end{lemma}

We conclude this section with a result obtained in \cite{pavel} which generalizes the famous Burnside $p^{\alpha}q^{\beta}$-Theorem (see \cite[Theorem 4.3.3]{go}).

\begin{theorem}[\cite{pavel}, Theorem 2.5]\label{prime} Let $k$ be a positive integer, let $\pi$ be a set consisting of at most two primes and let $G$ be a finite group in which all $\d_k$-commutators are $\pi$-elements. Then $G$ is soluble and $\d_k(G) \leq O_{\pi}(G)$.
\end{theorem}

\section{Proof of Theorem A}

We start with a  proposition which clarifies the connection between $\g_k(G)$ and the terms of the lower central series of a finite group $G$. 

\begin{proposition}\label{gammainfk}
If $G$ is a finite group, then $\g_k(G)=\gi(G)$ for every $k \geq 2$.
\end{proposition}

\begin{proof}
We argue by induction on $k$. Let $k=2$.
Since $G/\gi(G)$ is nilpotent, $[a,b] \in \gi(G)$ for every $a,b \in G$ of coprime orders. Hence $\g_2(G) \leq \gi(G)$. On the other hand, if $p,q \in \pi(G)$ with $p \neq q$, then $p$-elements and $q$-elements of $G/\g_2(G)$ commute. Therefore $G/\g_2(G)$ is nilpotent and $\gi(G) \leq \g_2(G)$.

Now, notice that $\gi(G)=[\gi(G),G]$ and that with a similar argument as before, it can be proved that 
\[[\gi(G),G]= \langle [a,b] \ | \ a \in \gi(G), b \in G, (|a|,|b|)=1 \rangle.\]
If $k \geq 2$, it follows that 
\begin{align*}
\g_{k+1}(G)&=\langle [a,b] \ | \ a \in \g_k(G), b \in G, (|a|,|b|)=1 \rangle \\
&= \langle [a,b] \ | \ a \in \gi(G), b \in G, (|a|,|b|)=1 \rangle \\
&= \gi(G),
\end{align*}
and we are done.
\end{proof}

The next lemma shows the role played by the condition imposed on the orders of $\g_k$-commutators.

\begin{lemma}\label{ddd}
Let $k\geq 2$ and let $G$ be a finite group in which $|ab|=|a||b|$ whenever $a$ and $b$ are $\gamma_k^*$-commutators of coprime orders. Assume that $N$ is a subgroup of $G$ normalized by  an element $x$ such that $(|N|,|x|)=1$. Then $[N, x]=1$.
\end{lemma}

\begin{proof}
Let $y \in N$. Then $g=[y, \ _{k-1} x^{-1}]$ is a  $\gamma_k^*$-commutator in $N$. Since $(|g|,|x|)=1$, it follows that
\[|gx|=|g| |x|.\]
However, 
\[gx=[y, \ _{k-2} x^{-1}]^{-1}x[y, \ _{k-2} x^{-1}],\]
hence $|g|=1$. Since $[N, \ _{k-1} x^{-1}]=[N, x^{-1}]$, it results that $[N,x]=~1$.
\end{proof}

The soluble case follows.

\begin{proposition}\label{soluGk}
If $G$ is a finite soluble group in which $|ab|=|a||b|$ whenever $a$ and $b$ are $\gamma_k^*$-commutators of coprime orders, for some $k\geq~2$, then $\gamma_k^*(G)$ is nilpotent. Moreover $G$ is metanilpotent.
\end{proposition}

\begin{proof}
Let $h$ be the Fitting height of $G$. Since $\gamma_k^*(G)=\gi(G)$ by Theorem \ref{gammainfk}, it follows that $\g_k(G)$ is nilpotent when $h \leq 2$.

Assume $h\geq 3$. Then, there exists a tower $P_1P_2P_3\ldots P_h$ of height $h$ in $G$. Since $P_2=[P_2,P_1]$, it follows that $$P_2=[P_2,\underbrace{P_1, \ldots,P_1}_{(k-1)\ times}].$$
Combining Lemma \ref{ddd} with the fact that $P_2$ is generated by $\gamma_k^*$-com\-mutators of $G$ of $p_2$-orders, we deduce that $P_3$ commutes with $P_2$. On the other hand, $[P_3,P_2]=P_3$, because $P_1P_2P_3\ldots P_h$ is a tower. This is a contradiction, and the proof is complete. 
\end{proof}

To succeed in completing the proof we need the following lemma.

\begin{lemma} \label{gk}
Let $G$ be a finite group such that $G=G'$ and let $q\in\pi(G)$. Then $G$ is generated by $\g_k$-commutators of $p$-power order for primes $p\neq q$.  
\end{lemma}

\begin{proof}
For each prime $p\in\pi(G)\setminus\{q\}$, denote by $N_p$ the subgroup generated by all $\g_k$-values of $p$-power order. Let us show that for each $p$ the Sylow $p$-subgroups of $G$ are contained in $N_p$. Suppose that this is false and choose $p$ such that a Sylow $p$-subgroup of $G$ is not contained in $N_p$. Since $N_p \lhd G$, we can pass to the quotient $G/N_p$ and we can assume that $N_p=1$. Now $G=G'$, so $G$ does not possess a normal $p$-complement. Therefore, by \cite[Theorem 7.4.5]{go}, it follows that $G$ has a $p$-subgroup $H$ and a $p'$-element $a\in N_G(H)$ such that $[H,a]\neq1$. 
Then, $$1\neq[H,a]=[H,\underbrace{a,\ldots,a}_{k-1\ times}]\leq N_p,$$ a contradiction. Therefore $N_p$ contains the Sylow $p$-subgroups of $G$. Let $T$ be the product of all $N_p$ for $p\neq q$. Then $G/T$ is a $q$-group, and so $G=T$ because $G=G'$. We conclude that $G$ can be generated by $\g_k$-values of $p$-power order for $p \neq q$, as desired.  

\end{proof}

To give the proof of Theorem A we will need some remarks on {minimal simple groups}. A minimal simple group is a nonabelian simple group whose proper subgroups are soluble. They have been classified by Thompson in his famous paper \cite{thompson}.

\begin{theorem}[\cite{thompson}, Corollary 1] \label{thomClass} Every finite minimal simple group is isomorphic
to one of the following groups:
\begin{enumerate}
\item[(1)] $\Psl(2, 2^p)$, where $p$ is any prime;
\item[(2)] $\Psl(2, 3^p)$, where $p$ is any odd prime;
\item[(3)] $\Psl(2, p)$, where $p > 3$ is any prime such that $p^2 + 1 \equiv 0 \mod 5$;
\item[(4)] $\Psl(3, 3)$;
\item[(5)] $\Sz(2^p)$, where $p$ is any odd prime.
\end{enumerate}
\end{theorem}

In light of Theorem \ref{thomClass}, we recall the following result from \cite{MT}.

\begin{proposition}[\cite{MT}]\label{maj}
Every finite minimal simple group $G$ contains a subgroup $H= A \rtimes T$ where $A$ is an elementary abelian $2$-group and $T$ is a subgroup of odd order such that $C_A(T)=1$. 
\end{proposition}

We are now in position to give the proof of our first main result.
 
\begin{proof}[\bf Proof of Theorem A] Since the necessary condition for $\g_k(G)$ to be nilpotent is obviously satisfied, we only need to prove the sufficient one.

Suppose that the theorem is false and let $G$ be a counterexample of minimal order. In view of Proposition \ref{soluGk} $G$ is not soluble while all proper subgroups of $G$ so are. Indeed, if $H <G$ then $\g_k(H)$ is nilpotent by minimality of $G$; moreover, $H/\g_k(H)=H/\gi(H)$ is nilpotent, thus $H$ is soluble.
Therefore $G=G'$. Let $R$ be the soluble radical of $G$. It follows that $G/R$ is a nonabelian simple group. Moreover, by Proposition \ref{soluGk}, $R$ is metanilpotent. We want to show that $R$ coincides with $Z(G)$, which is obvious if $R=1$. Suppose $R\neq1$ and choose $q\in\pi(F(G))$. According to Lemma \ref{gk}, $G$ is generated by the $\g_k$-commutators of $p$-power order for primes $p\neq q$. Let $Q$ be the Sylow $q$-subgroup of $F(G)$. By Lemma \ref{ddd} $[Q,x]=1$, for every $\g_k$-commutator $x$ of $q'$-order. Therefore $Q\leq Z(G)$. This happens for each choice of $q\in\pi(F(G))$, so that $F(G)= Z(G)$.  

Now, since $R$ is metanilpotent, there exists a positive integer $s$ such that $\gamma_s(R)$ is nilpotent. Therefore, if $x \in R$ and $y \in G$, we have
\[[[y,x], \underbrace{x,\ldots,x}_{s-1 \ times}]\in\gamma_s(R)\leq F(G)=Z(G).\] 
Consequently, $[y, \ _{s+1} x]=1$, and so every element $x$ is Engel. Hence $R\leq F(G)$ by Theorem \cite[12.3.7]{rob}. Thus $R=Z(G)$ and $G$ is quasisimple. 

We claim that $G$ contains a $\g_k$-commutator  $a$ such that $a$ is a $2$-element and $a$ has order 2 modulo $Z(G)$. Since $G/Z(G)$ is minimal simple, Proposition \ref{maj} implies the existence of a subgroup \[H/Z(G) = A/Z(G) \rtimes T/Z(G)\] where $A/Z(G)$ is an elementary abelian $2$-group and $T/Z(G)$ is a group of odd order such that $C_{A/Z(G)}(T/Z(G))=1$. 
Since $A$ is nilpotent, its Sylow $2$-subgroup $P$ is normal in $H$ and $A=PZ(G)$. Therefore, \[A/Z(G) = [A/Z(G), \ _{k-1} T/Z(G)],\] and there exists $[x, \ _{k-1} y]Z(G)$ in $A/Z(G)$ with $x$ in $P$ and $yZ(G)$ in $T/Z(G)$. Moreover, since $\langle y \rangle Z(G)$ is nilpotent, we have $\langle y \rangle Z(G)=WZ(G)$ with $W$ of odd order. Hence we may suppose that $y$ has odd order. If $g=[x, \ _{k-1} y]z$, for some $z \in Z(G)$, then $[g,y]$ is a $\gamma_k^*$-commutator of $2$-power order with order $2$ modulo the center.

Finally, fix an element $a$ with the above properties. Since $G/Z(G)$ is a nonabelian simple group, it follows from \cite[Theorem 3.8.2]{go} that there exists an element $t\in G$ such that the order of $[a,t]$ is odd. On the one hand, it is clear that $a$ inverts $[a,t]$. On the other hand, by Lemma \ref{ddd}, $a$ centralizes $[a,t]$. This is a contradiction.
\end{proof}

 In particular, Theorem A applies to $\gi(G)$, as showed by the following two corollaries.

\begin {corollary}
The nilpotent residual $\gi(G)$ of a finite group $G$ is nilpotent if and only if $|ab|=|a| |b|$ whenever $a$ and $b$ are $\gamma_k^*$-commutators of coprime orders, for some $k \geq 2$.
\end {corollary}

\begin {corollary}\label{corollarioE}
Let $G$ be a finite group. Then the nilpotent residual $\gamma_{\infty}(G)$ is nilpotent if and only if $|ab|=|a| |b|$ whenever $a$ and $b$ are $\d_1$-commutators of coprime orders.
\end {corollary}

 As an immediate consequence of Theorem~\ref{teoMT} and Corollary~\ref{corollarioE} we obtain the following corollary.

\begin {corollary}
Let $G$ be a finite group. The following properties are equivalent:
\begin{enumerate}
\item[(1)] $|ab|=|a| |b|$ whenever $a$ and $b$ are $\d_1$-commutators of coprime orders;
\item[(2)] the nilpotent residual of $G$ is nilpotent;
\item[(3)] the Fitting height of $G$ is at most $2$;
\item[(4)] there exists a simple commutator word $w=[x_{i_1}, \ldots, x_{i_k}]$ with $i_1 \neq i_j$ for every $j \in \{2, \ldots, k\}$ such that $|xy|=|x||y|$ for any $w$-values $x,y\in G$ of coprime orders.
\end{enumerate}
\end {corollary}

\section{Proof of Theorem B}

We start with the following lemma.
\begin{lemma}\label{commfond}
Let $k$ be a positive integer and let $G$ be a finite group in which $|ab|=|a| |b|$ whenever $a$ and $b$ are $\d_k$-commutators such that $(|a|,|b|)=1$. Let $x$ be a $\d_k$-commutator of $G$ and let $N$ be a subgroup of $G$ normalized by $x$. If $(|N|,|x|)=1$, then $[N, x]=1$.
\end{lemma}

\begin{proof}
By Lemma \ref{uuu}, if $y \in N$ the element $[y,{}_k x]$ is a $\d_k$-commutator in $N$, whose order is prime to the order of $x$. Hence, since $x^{-1}$ is still a $\d_k$-commutator,
\[\displaystyle{|[y,{}_k x] x^{-1}|=|[y,{}_k x]| |x^{-1}|}.\]
However,  $[y,{}_k x] x^{-1}$ is a conjugate of $x^{-1}$, and so $[y,{}_k x]=1$. Therefore, $[N,x]=[N,{}_k x]=1$, and the lemma is proved.
\end{proof}

The following result shows that in the soluble case we achieve the nilpotency of $\d_k(G)$ in weaker hypothesis; in fact, we impose the condition on the orders of $\d_k$-commutators, without considering the powers.

\begin{proposition}\label{main3}
Let $G$ be a finite soluble group and let $k \geq 1$. If $|ab|=|a| |b|$ whenever $a$ and $b$ are $\d_k$-commutators such that $(|a|,|b|)=1$, then $\d_k(G)$ is nilpotent.
\end{proposition}

\begin{proof}
Let $h=h(G)$. By Theorem \ref{fittingdelta}, when $h \leq k+1$ then $\d_{k+1}(G)=1$. Therefore, Theorem \ref{deltagen} implies that $\d_{k+1}(G)=\gamma_{\infty}(\d_k(G))$ and so $\d_k(G)$ is nilpotent. 

Now  assume $h \geq k+2$. Thus there exists a tower \[ P_1 \cdots P_{k+2} \cdots P_h\] of $G$ of height $h$. By Lemma \ref{coprimetower},  $P_{k+1}$ is generated by $\d_k$-commutators in $P_{k+1}$. Combining Lemma \ref{commfond} with the fact that $P_{k+1}$ is generated by $\d_k$-commutators of $p_{k+1}$-orders, we deduce that $P_{k+2}$ commutes with $P_{k+1}$. On the other hand, $[P_{k+2},P_{k+1}]=P_{k+2}$, because $P_1\cdots P_h$ is a tower. This is a contradiction, and the proof is complete. 
\end{proof}

Now, we need a result concerning the generators of a minimal nonsoluble group. Recall that a group $G$ is {\it minimal nonsoluble} if every proper subgroup of $G$ is soluble, but $G$ is not soluble.  

\begin{lemma}\label{genDk}
Let $G$ be a finite minimal nonsoluble group. If $q \in \pi(G)$, then $G$ can be generated by powers of $\d_k$-commutators of $q'$-order.
\end{lemma}

\begin{proof}
Let $T$ be the subgroup of $G$ generated by all powers of $\d_k$-commutators of $q'$-order. We want to prove that $T=G$. Suppose that $T < G$. Then $T$ is soluble and, being normal, we can consider the quotient $G/T$. If $gT$ is a $\d_k$-commutator of $G/T$, we have $|gT|=q^{\alpha}m$, where $q$ does not divide $m$. Hence $(g)^{q^{\alpha}} \in T$, and so every $\d_k$-commutator of $G/T$ has $q$-order. Therefore, by Theorem \ref{prime}, $G/T$ is soluble. In particular $G$ is soluble, a contradiction.
\end{proof}

In addition, we will need the following fact, already stressed in \cite{pavel}, about elements of order $2$ in minimal simple groups. Recall that an element $g$ of a group $G$ is strongly real if there exists an element $x\in G$ of order $2$ such that $g^x=g^{-1}$.

\begin{proposition}\label{invstar}
In a minimal simple group all elements of order $2$ are $\d_n$-commutators,  for every $n \geq 0$. 

\end{proposition}

\begin{proof}
Let $G$ be a minimal simple group. It is well-known that all all elements of order $2$ are conjugate in $G$. By Theorem~\ref{thomClass} we have to analyze five cases.\\

Assume that $G=\Psl(2,3^p)$. By Proposition~\ref{maj}, $G$ contains the alternating group of degree $4$ with Frobenius complement $T$ and Frobenius kernel $A$. Then $T$ is normalized by a strongly real element $u$ of prime order dividing either $(3^p-1)/2$ or $(3^p+1)/2$. If $K$ is the subgroup generated by $u$, then $T \rtimes K$ and $C_T(K)=1$. Let $x$ be a generator of $T$. Then
\[x=[x_1, \ _n u],\]
for some $x_1 \in T$ and for every $n \geq 1$. Assume that $j$ is an element of $G$ of order $2$ such that $u^j=u^{-1}$. Hence, if $y \in K$ we have
\[y=[y_1, \ _n j],\]
 for some $y_1 \in K$ and for every $n \geq 1$. The same happens for $A$, that is, when $a \in A$ there exists $a_1 \in A$ such that
\[a=[a_1, \ _n x],\] 
for any $n \geq 1$. We prove that $a,x,y$  are  $\d_n$-commutators for every $n \geq 0$. Indeed, since the case $n=0$ is obvious, suppose that $n >0$ and assume that $a, x, y$ are $\d_{n-1}$-commutators. By Lemma \ref{uuu}, $x$ is a $\d_n$-commutator because $x=[x_1, \ _n u]$. Hence, applying again Lemma \ref{uuu}, $a$ is a $\d_{n+1}$-commutator and so are all elements of order $2$ of $G$. Finally, Lemma~\ref{uuu} implies that $y$ is a  $\d_n$-commutator.  \\

Now assume that $G$ is not isomorphic to $\Psl(2,3^p)$. By Proposition~\ref{maj} $G$ contains a subgroup $H=A \rtimes T$ with $A$ elementary abelian $2$-group and $T$ a group of odd order such that $C_A(T)=1$. We can assume $T$ cyclic. If $T$ contains a strongly real element, then the elements of $T$ and all the elements of order $2$ are $\d_n$-commutators for every $n \geq 0$. Assume that this is the case and let $T= \langle x \rangle$ where $x$ is a strongly real element such that $C_A(x)=1$. Since $x$ is strongly real, there exists $j \in G$ of order $2$ such that $x^j=x^{-1}$. Hence $T=[T, j]$ and for every $y\in T$ there exists $y_1 \in T$ such that $y=[y_1, \ _n j]$. Then, for every $n \geq 1$ and every involution $a \in A$ we have
\[a=[b, \ _{n-1}x]\]
for some $b \in A$. We show that $a$ and $x$ are $\d_n$-commutators for every $n \geq 0$. Indeed, suppose that $n \geq 1$ and $x$ is a $\d_{n-1}$-commutator, being clear the case $n=0$. Since $a=[b, \ _{n-1} x]$, Lemma \ref{uuu} implies that $a$ is a $\d_n$-commutator. Since all elements of order $2$ in $G$ are conjugate,  they are $\d_n$-commutators. Therefore, applying again Lemma \ref{uuu} it follows that $x$ is a $\d_{n+1}$-commutator.\\

Let $G$ be isomorphic to the Suzuki group $\Sz(q)$ with $q=2^p$ and $p$ an odd prime. From \cite[Theorem 9]{suzuki2}, $G$ contains a Frobenius group $H= A \rtimes T$ of order $q^2(q-1)$, where A is a Sylow $2$-subgroup of $G$ and $T$ is a cyclic subgroup of order $q-1$ normalizing $A$. Let $x$ be a generator of $T$. By \cite[Proposition 3]{suzuki2}, the normalizer of $T$ in $G$ is a dihedral group $T \rtimes K$ of order $2(q-1)$, and a generator $j \in K$ is such that $x^j=x^{-1}$. \\ 

Let $G$ be isomorphic to $\Psl(2,2^p)$, with $p$ any prime. By \cite[Theorem 6.25]{suzuki} $G$ contains a subgroup $H=Q \rtimes L$ which is the semidirect product of a Sylow $2$-subgroup $Q$ order $2^p$  and a cyclic group $T$ of order $2^p-1$. Moreover, the normalizer of $T$ in $G$ is dihedral of order $2(2^p-1)$, and $G$ contains strongly real elements of order dividing $2^p-1$ and $2^p+1$.\\

If $G$ is isomorphic to $\Psl(3,3)$, then $G$ contains a copy of the alternating group of degree $4$, with cyclic Frobenius complement $T= \langle x \rangle$, and $x$ strongly real. \\

Finally, if $G$ is isomorphic to $\Psl(2,p)$, with $p$ odd prime, then $G$ contains a copy of the alternating group  of degree $4$, with cyclic Frobenius complement $T$ of order $3$. A generator of $T$ is strongly real. The proof is concluded.
\end{proof}

Now we are ready to prove our second main result.

\begin{proof}[\bf Proof of Theorem B]
Clearly if $\d_k(G)$ is nilpotent, then $|ab| = |a| |b|$ for any power of $\d_k$-commutators $a,b \in G$ of coprime orders. So we only need to prove the converse. 

Suppose that the theorem is false and let $G$ be a counterexample of minimal order. In view of Proposition \ref{main3}, $G$ is not soluble while all proper subgroups of $G$ so are. In fact, if $H < G$ then $\d_k(H)$ is nilpotent, and so $\delta_{k+1}^*(H)=1$. Thus, $H$ is soluble by Theorem \ref{fittingdelta}. It follows that $G = G'$. 

If $R$ is the soluble radical of $G$, then $G/R$ is a nonabelian simple group. Now we want to prove that $G/ Z(G)$ is simple. We will do it proving that $R= Z(G)$, which is obvious if $R=1$. Suppose $R \neq 1$ and choose $q \in \pi (F(G))$. According to Lemma \ref{genDk} $G$ is generated by power of $\d_k$-commutators of $q'$-order. Let $Q$ be the Sylow $q$-subgroup of $F(G)$. By Lemma \ref{commfond} $[Q, x] = 1$, for every power of $\d_k$-commutator $x$ of $q'$-order. Therefore $Q \leq Z(G)$, and this happens for each choice of $q \in \pi(F(G))$. Hence, we have $F(G) = Z(G)$.

Now suppose that $R$ is not nilpotent,  that is $\gamma_{\infty}(R) \neq 1$. Since $R$ is soluble, the lower Fitting series of $R$, 
\[
R=N_0(R) \geq N_1(R) \geq \cdots \geq N_t(R)=1,
\]
has length $t  \geq 2$. It follows that $N_{t-1}(R)$ is a characteristic nilpotent subgroup of $R$, and so $N_{t-1}(R) \leq F(G)=Z(G)$. Therefore $N_{t-2}(R)$ is nilpotent and $N_{t-1}(R)=\gamma_{\infty}(N_{t-2}(R))=1$, a contradiction. Thus $R$ is nilpotent and $R\leq F(G) = Z(G)$. Therefore $R=Z(G)$ and $G$ is quasisimple.

We claim that $G$ contains a power of $\d_k$-commutator $a$ of $2$-order such that $a$ has order $2$ modulo $Z(G)$. Since $G/Z(G)$ is a minimal simple group, by Lemma \ref{invstar}, there exists a $\d_k$-commutator $gZ(G)$ of order $2$. Then there exists a $\d_k$-commutator $y$ in $G$ such that $gZ(G)=yZ(G)$. Hence, we can choose $a$ to be a power of $y$ such that $a$ is a $2$-element and $a$ has order 2 modulo $Z(G)$.

Let $a$ be an element of $G$ with the above properties. Since $G/Z(G)$ is a nonabelian simple group, it follows from \cite[Theorem 3.8.2]{go} that there exists an element $t\in G$ such that the order of $[a,t]$ is odd. On the one hand, since $1=[a^2,t]=[a,t]^a[a,t]$, $a$ inverts $[a,t]$. On the other hand, by Lemma \ref{commfond}, $a$ centralizes $[a,t]$. This is a contradiction.
\end{proof}

Finally, the following result is an immediate consequence of Theorem~B.

\begin {corollary}
Let $G$ be a finite group and let $k \geq 1$. Assume that $|ab|=|a| |b|$ whenever $a$ and $b$ are powers of $\d_k$-commutators such that $(|a|,|b|)=1$. Then $\d_k(G)$ is nilpotent and the Fitting height of $G$ is at most $k+1$.
\end {corollary}

\section{Acknowledgment}
 The authors wish to thank professor Pavel Shumyatsky for interesting discussions. The work of the first author was partially supported by FAPDF/Brazil, while the second author was supported by the ``National Group for Algebraic and Geometric Structures, and their Applications" (GNSAGA - INdAM). Moreover, this study was carried out during the second author's visit to the University of Brasilia. He wishes to thank the Department of Mathematics for the excellent hospitality.

\end{document}